\documentclass[12pt]{article}
\usepackage{amssymb}
\usepackage{amsthm}
\usepackage{amsmath}
\usepackage{graphicx}
\usepackage{enumerate}

\newtheorem{theorem}{Theorem}[section]
\newtheorem{definition}[theorem]{Definition}
\newtheorem{lemma}[theorem]{Lemma}

\newtheorem{example}[theorem]{Example}
\newtheorem{corollary}[theorem]{Corollary}
\title{Residuation in twist products and pseudo-Kleene posets}
\author{Ivan~Chajda and Helmut~L\"anger}
\date{}
\begin{document}

\footnotetext{Support of the research by the Austrian Science Fund (FWF), project I~4579-N, and the Czech Science Foundation (GA\v CR), project 20-09869L, entitled ``The many facets of orthomodularity'', as well as by \"OAD, project CZ~02/2019, entitled ``Function algebras and ordered structures related to logic and data fusion'', and, concerning the first author, by IGA, project P\v rF~2020~014, is gratefully acknowledged.}

\maketitle

\begin{abstract}
M.~Busaniche, R.~Cignoli, C.~Tsinakis and A.~M.~Wille showed that every residuated lattice induces a residuation on its full twist product. For their construction they used also lattice operations. We generalize this problem to left-residuated groupoids which need not be lattice-ordered. Hence, for the full twist product we cannot use the same construction. We present another appropriate construction which, however, does not preserve commutativity and associativity of multiplication. Hence we introduce so-called operator residuated posets to obtain another construction which preserves the mentioned properties, but the results of operators on the full twist product need not be elements, but may be subsets. We apply this construction also to restricted twist products and present necessary and sufficient conditions under which we obtain a pseudo-Kleene operator residuated poset.
\end{abstract}



\section{Introduction}

M.~Busaniche and R.~Cignoli (\cite{BC}) as well as by C.~Tsinakis and A.~M.~Wille (\cite{TW}) showed that if $(L,\leq,\cdot,\rightarrow,1)$ is a residuated lattice then $\cdot$ and $\rightarrow$ can be used to define binary operations $\odot$ and $\Rightarrow$ on the full twist product of $(L,\leq)$ such that the resulting structure becomes a residuated lattice again. For the construction of such operations $\odot$ and $\Rightarrow$ they used the lattice operations $\vee$ and $\wedge$. When going from lattices to residuated posets, the natural question arises whether also in this case the corresponding twist product can be equipped with certain operations $\odot$ and $\Rightarrow$ (without using lattice operations) such that the resulting structure is residuated again. We solve this problem in the positive. We define suitable operations $\odot$ and $\Rightarrow$ on the full twist product such that the arising structure becomes a left-residuated groupoid again. Unfortunately, this construction does preserve neither commutativity nor associativity of the original structure $(Q,\leq,\cdot,\rightarrow,1)$. Hence, we try another approach where instead of operations we use certain operators $\odot$ and $\Rightarrow$ in such a way that the full twist product becomes an operator residuated poset and commutativity as well as associativity of the original operation $\cdot$ are preserved. As the authors already showed in \cite{CL}, any poset $\mathbf Q=(Q,\leq)$ can be embedded into a pseudo-Kleene poset $(P_a(\mathbf Q),\leq,{}')$ where $(P_a(\mathbf Q),\leq)$ is a certain subposet of the full twist product of $\mathbf Q$. This motivated us to investigate whether our construction of the operators $\odot$ and $\Rightarrow$ can be extended also to this case, i.e.\ whether we can determine for a bounded commutative residuated monoid $(Q,\leq,\cdot,\rightarrow,0,1)$ a corresponding pseudo-Kleene poset which is operator residuated and into which $\mathbf Q$ can be embedded. We characterize those left-residuated posets for which our construction is possible.

\section{Preliminaries}

The concept of a Kleene lattice (alias Kleene algebra) was introduced by J.~A.~Kalman (\cite K), see also \cite{Ci}. Recall that a {\em Kleene lattice} is a distributive lattice $\mathbf L=(L,\vee,\wedge,{}')$ with an involution $'$ satisfying the so-called {\em normality condition}, i.e.\ the identity
\[
x\wedge x'\leq y\vee y'.
\]
This concept was generalized by the first author in \cite{Ch}: $\mathbf L$ is called a {\em pseudo-Kleene lattice} if it satisfies the above identity, but it need not be distributive.

Let $(P,\leq)$ be a poset, $a,b\in P$ and $A,B\subseteq P$. Then the lower cone $L(A)$ of $A$ and the upper cone $U(A)$ of $A$ are defined as follows:
\begin{align*}
L(A) & :=\{x\in P\mid x\leq A\}, \\
U(A) & :=\{x\in P\mid x\geq A\}.
\end{align*}
Here $x\leq A$ means $x\leq y$ for all $y\in A$ and, similarly, $x\geq A$ means $x\geq y$ for all $y\in A$. The expression $A\leq B$ means $x\leq y$ for all $x\in A$ and $y\in B$. Instead of $L(\{a,b\})$ and $L(U(A))$ we simply write $L(a,b)$ and $LU(A)$, respectively. Analogously, we proceed in similar cases. Let $\max A$ denote the set of all maximal elements of $(A,\leq)$. A unary operation $'$ on $P$ is called
\begin{itemize}
\item {\em antitone} if $x,y\in P$ and $x\leq y$ imply $y'\leq x'$,
\item an {\em involution} if it satisfies the identity $x''\approx x$.
\end{itemize}
The concept of a pseudo-Kleene lattice was generalized by the authors in \cite{CL} for posets as follows:

A {\em pseudo-Kleene poset} is a poset $\mathbf P=(P,\leq,{}')$ with an antitone involution satisfying the condition
\[
L(x,x')\leq U(y,y')
\]
for all $x,y\in P$. A {\em Kleene poset} is a distributive pseudo-Kleene poset. Recall that a {\em poset} $(P,\leq)$ is called {\em distributive} if it satisfies one of the following equivalent LU-identities:
\begin{align*}
L(U(x,y),z) & \approx LU(L(x,z),L(y,z)), \\
U(L(x,y),z) & \approx UL(U(x,z),U(y,z)).
\end{align*}
In \cite{CL} it was shown that an arbitrary poset can be embedded into a pseudo-Kleene poset by means of the so-called {\em twist construction}:

The {\em full twist product} of a poset $\mathbf Q=(Q,\leq)$ is the poset $(Q^2,\leq)$ where
\[
(x,y)\leq(z,v)\text{ if and only if }x\leq z\text{ and }v\leq y
\]
for all $(x,y),(z,v)\in Q^2$. We have
\begin{align*}
L((x,y),(z,v)) & =L(x,z)\times U(y,v), \\
U((x,y),(z,v)) & =U(x,z)\times L(y,v)
\end{align*}
for all $(x,y),(z,v)\in Q^2$.

\section{Left-residuated groupoids}

We will investigate when a residuated poset can be transferred to a residuated full twist product. For this purpose we will use the twist construction. For residuated lattices such a transfer was already published in \cite{BC} by using a construction developed in \cite{TW}.

From now on, let $(Q,\leq,\cdot,\rightarrow,1)$ denote a poset with constant $1$ endowed with two binary operations $\cdot$ and $\rightarrow$. For our next investigations, consider the following conditions.
\begin{enumerate}[(1)]
\item $x\leq y$ implies $z\cdot x\leq z\cdot y$ (right-isotony),
\item $x\leq y$ implies $x\cdot z\leq y\cdot z$ (left-isotony),
\item $x\cdot y\leq z$ if and only if $x\leq y\rightarrow z$ (left-adjointness),
\item $x\leq y$ implies $z\rightarrow x\leq z\rightarrow y$,
\item $x\leq y$ implies $y\rightarrow z\leq x\rightarrow z$,
\item $x\cdot1\approx x$,
\item $x\cdot y\leq x,y$
\end{enumerate}
for all $x,y,z\in Q$.

The above mentioned conditions are related as shown in the following Lemmas.

\begin{lemma}\label{lem2}
For $(Q,\leq,\cdot,\rightarrow,1)$ the following hold:
\begin{enumerate}[{\rm(i)}]
\item {\rm(1)} and {\rm(3)} imply {\rm(5)}.
\item If $\cdot$ is commutative then {\rm(1)} and {\rm(6)} imply {\rm(7)}.
\end{enumerate}
\end{lemma}

\begin{proof}
Let $a,b,c\in Q$.
\begin{enumerate}[(i)]
\item If $a\leq b$ then every one of the following statements implies the next one:
\begin{align*}
         b\rightarrow c & \leq b\rightarrow c, \\
(b\rightarrow c)\cdot b & \leq c, \\
(b\rightarrow c)\cdot a & \leq c, \\
         b\rightarrow c & \leq a\rightarrow c.
\end{align*}
\item We have $a\cdot b\leq a\cdot1=a$ and $a\cdot b=b\cdot a\leq b$.
\end{enumerate}
\end{proof}

\begin{lemma}\label{lem3}
Assume $(Q,\leq,\cdot,\rightarrow,1)$ with associative $\cdot$ to satisfy {\rm(2)} and {\rm(3)}. Then it satisfies
\begin{enumerate}
\item[{\rm(8)}] $(x\cdot y)\rightarrow z\approx x\rightarrow(y\rightarrow z)$.
\end{enumerate}
\end{lemma}

\begin{proof}
Let $a,b,c\in Q$. Then every of the following statements implies the next one:
\begin{align*}
                  (a\cdot b)\rightarrow c & \leq(a\cdot b)\rightarrow c, \\
 ((a\cdot b)\rightarrow c)\cdot(a\cdot b) & \leq c, \\
(((a\cdot b)\rightarrow c)\cdot a)\cdot b & \leq c, \\
         ((a\cdot b)\rightarrow c)\cdot a & \leq b\rightarrow c, \\
                  (a\cdot b)\rightarrow c & \leq a\rightarrow(b\rightarrow c).
\end{align*}
Moreover, every one of the following statements implies the next one:
\begin{align*}
                  a\rightarrow(b\rightarrow c) & \leq a\rightarrow(b\rightarrow c), \\
         (a\rightarrow(b\rightarrow c))\cdot a & \leq b\rightarrow c, \\
((a\rightarrow(b\rightarrow c))\cdot a)\cdot b & \leq(b\rightarrow c)\cdot b, \\
 (a\rightarrow(b\rightarrow c))\cdot(a\cdot b) & \leq(b\rightarrow c)\cdot b.
\end{align*}
Together with $(b\rightarrow c)\cdot b\leq c$ which follows from $b\rightarrow c\leq b\rightarrow c$ we obtain $(a\rightarrow(b\rightarrow c))\cdot(a\cdot b)\leq c$ which implies $a\rightarrow(b\rightarrow c)\leq(a\cdot b)\rightarrow c$.
\end{proof}

Now we define one of our main concepts.

\begin{definition}
$(Q,\leq,\cdot,\rightarrow,1)$ is called a {\em left-residuated groupoid} if it satisfies {\rm(3)} and {\rm(6)}. It is called
\begin{itemize}
\item {\em bounded} if $(Q,\leq)$ is bounded {\rm(}$0$ is the bottom and $1$ the top element{\rm)},
\item {\em commutative} if $\cdot$ is commutative,
\item {\em associative} if $\cdot$ is associative.
\end{itemize}
A {\em commutative residuated monoid} is a commutative and associative left-residuated \\
groupoid.
\end{definition}

An example of a bounded residuated monoid which is not a lattice is as follows:

\begin{example}\label{ex1}
The poset visualized in Figure~1:

\vspace*{-2mm}

\begin{center}
\setlength{\unitlength}{7mm}
\begin{picture}(6,12)
\put(3,1){\circle*{.3}}
\put(3,3){\circle*{.3}}
\put(1,5){\circle*{.3}}
\put(3,5){\circle*{.3}}
\put(5,5){\circle*{.3}}
\put(1,7){\circle*{.3}}
\put(3,7){\circle*{.3}}
\put(5,7){\circle*{.3}}
\put(3,9){\circle*{.3}}
\put(3,11){\circle*{.3}}
\put(3,3){\line(-1,1)2}
\put(3,3){\line(1,1)2}
\put(3,9){\line(-1,-1)2}
\put(3,9){\line(1,-1)2}
\put(1,5){\line(0,1)2}
\put(1,5){\line(1,1)2}
\put(3,5){\line(-1,1)2}
\put(3,5){\line(1,1)2}
\put(5,5){\line(-1,1)2}
\put(5,5){\line(0,1)2}
\put(3,1){\line(0,1){10}}
\put(2.85,.25){$0$}
\put(3.4,2.85){$a$}
\put(.3,4.85){$b$}
\put(3.4,4.85){$c$}
\put(5.4,4.85){$d$}
\put(.3,6.85){$e$}
\put(3.4,6.85){$f$}
\put(5.4,6.85){$g$}
\put(3.4,8.85){$h$}
\put(2.85,11.4){$1$}
\put(2.2,-.75){{\rm Fig.\ 1}}
\end{picture}
\end{center}

\vspace*{4mm}

together with the operations given by
\[
\begin{array}{c|cccccccccc}
\cdot & 0 & a & b & c & d & e & f & g & h & 1 \\
\hline
  0   & 0 & 0 & 0 & 0 & 0 & 0 & 0 & 0 & 0 & 0 \\
	a   & 0 & 0 & 0 & 0 & 0 & 0 & 0 & 0 & 0 & a \\
	b   & 0 & 0 & a & 0 & 0 & a & a & 0 & a & b \\
	c   & 0 & 0 & 0 & a & 0 & a & 0 & a & a & c \\
	d   & 0 & 0 & 0 & 0 & a & 0 & a & a & a & d \\
	e   & 0 & 0 & a & a & 0 & a & a & a & a & e \\
	f   & 0 & 0 & a & 0 & a & a & a & a & a & f \\
	g   & 0 & 0 & 0 & a & a & a & a & a & a & g \\
	h   & 0 & 0 & a & a & a & a & a & a & a & h \\
	1   & 0 & a & b & c & d & e & f & g & h & 1
\end{array}
\quad\quad\quad
\begin{array}{c|cccccccccc}
\rightarrow & 0 & a & b & c & d & e & f & g & h & 1 \\
\hline
     0      & 1 & 1 & 1 & 1 & 1 & 1 & 1 & 1 & 1 & 1 \\
		 a      & h & 1 & 1 & 1 & 1 & 1 & 1 & 1 & 1 & 1 \\
		 b      & g & h & 1 & h & h & 1 & 1 & h & 1 & 1 \\
		 c      & f & h & h & 1 & h & 1 & h & 1 & 1 & 1 \\
		 d      & e & h & h & h & 1 & h & 1 & 1 & 1 & 1 \\
		 e      & d & h & h & h & h & 1 & h & h & 1 & 1 \\
		 f      & c & h & h & h & h & h & 1 & h & 1 & 1 \\
		 g      & b & h & h & h & h & h & h & 1 & 1 & 1 \\
		 h      & a & h & h & h & h & h & h & h & 1 & 1 \\
		 1      & 0 & a & b & c & d & e & f & g & h & 1
\end{array}
\]
is a bounded commutative residuated monoid which is not a lattice.
\end{example}

The next lemma shows some elementary properties of left-residuated groupoids.

\begin{lemma}\label{lem1}
The following hold:
\begin{enumerate}[{\rm(i)}]
\item Every left-residuated groupoid $(Q,\leq,\cdot,\rightarrow,1)$ satisfies {\rm(2)}, {\rm(4)} and
\begin{enumerate}
\item[{\rm(9)}] $1\rightarrow x\approx x$.
\end{enumerate}
\item If $(Q,\leq,\cdot,\rightarrow,1)$ satisfies {\rm(5)} and {\rm(9)} then it satisfies
\begin{enumerate}
\item[{\rm(10)}] $x\leq y\rightarrow x$.
\end{enumerate}
\end{enumerate}
\end{lemma}

\begin{proof}
Let $a,b,c\in Q$.
\begin{enumerate}[(i)]
\item
\begin{enumerate}
\item[(2)] If $a\leq b$ then every one of the following statements implies the next one:
\begin{align*}
b\cdot c & \leq b\cdot c, \\
       b & \leq c\rightarrow(b\cdot c), \\
       a & \leq c\rightarrow(b\cdot c), \\
a\cdot c & \leq b\cdot c.
\end{align*}
\item[(4)] If $a\leq b$ then every one of the following statements implies the next one:
\begin{align*}
         c\rightarrow a & \leq c\rightarrow a, \\
(c\rightarrow a)\cdot c & \leq a, \\
(c\rightarrow a)\cdot c & \leq b, \\
         c\rightarrow a & \leq c\rightarrow b.
\end{align*}
\item[(9)] We have $1\rightarrow a\leq1\rightarrow a$ implies $1\rightarrow a=(1\rightarrow a)\cdot1\leq a$, and $a\cdot1\leq a$ implies $a\leq1\rightarrow a$.
\end{enumerate}
\item We have $a=1\rightarrow a\leq b\rightarrow a$.
\end{enumerate}
\end{proof}

Now we show that every left-residuated groupoid naturally induces a left-residuated groupoid on its full twist product.

\begin{theorem}\label{th2}
Let $(Q,\leq,\cdot,\rightarrow,1)$ be a poset with binary operations $\cdot$ and $\rightarrow$ and a constant $1$, let $a,b\in Q$ and $f,g$ be surjective mappings from $Q^2$ to $Q$ satisfying $f(a,b)=g(a,b)=1$ and consider the full twist product $(Q^2,\leq,\odot,\Rightarrow,(a,b))$ of $(Q,\leq)$ with binary operations $\odot$ and $\Rightarrow$ defined by
\begin{align*}
      (x,y)\odot(z,v) & :=(x\cdot f(z,v),g(z,v)\rightarrow y), \\
(x,y)\Rightarrow(z,v) & :=(f(x,y)\rightarrow z,v\cdot g(x,y))
\end{align*}
for all $(x,y),(z,v)\in Q^2$ and the constant $(a,b)$. Then $(Q,\leq,\cdot,\rightarrow,1)$ is a left-residuated groupoid if and only if $(Q^2,\leq,\odot,\Rightarrow,$ $(a,b))$ has this property.
\end{theorem}

\begin{proof}
We investigate when $(Q^2,\leq,\odot,\Rightarrow,(a,b))$ satisfies (3) and (6). The following are equivalent:
\begin{align*}
& (Q^2,\leq,\odot,\Rightarrow,(a,b))\text{ satisfies (3)}, \\
& (x,y)\odot(z,v)\leq(t,w)\text{ is equivalent to }(x,y)\leq(z,v)\Rightarrow(t,w), \\
& (x\cdot f(z,v),g(z,v)\rightarrow y)\leq(t,w)\text{ is equivalent to }(x,y)\leq(f(z,v)\rightarrow t,w\cdot g(z,v)), \\
& (x\cdot f(z,v)\leq t\text{ and }w\leq g(z,v)\rightarrow y)\text{ is equivalent to} \\
& \quad\quad(x\leq f(z,v)\rightarrow t\text{ and }w\cdot g(z,v)\leq y), \\
& (Q,\leq,\cdot,\rightarrow,1)\text{ satisfies (3)}.
\end{align*}
Moreover, the following are equivalent:
\begin{align*}
& (Q^2,\leq,\odot,\Rightarrow,(a,b))\text{ satisfies (6)}, \\
& (x,y)\odot(a,b)\approx(x,y), \\
& (x\cdot f(a,b),g(a,b)\rightarrow y)\approx(x,y), \\
& (x\cdot1,1\rightarrow y)\approx(x,y), \\
& x\cdot1\approx x\text{ and }1\rightarrow y\approx y, \\
& (Q,\leq,\cdot,\rightarrow,1)\text{ satisfies (6) and (7)}.
\end{align*}
Now Lemma~\ref{lem1} completes the proof.
\end{proof}

\begin{corollary}
Let $(Q,\leq,\cdot,\rightarrow,1)$ be a poset with binary operations $\cdot$ and $\rightarrow$ and a constant $1$ and consider the full twist product $(Q^2,\leq,\odot,\Rightarrow,(1,1))$ of $(Q,\leq)$ with binary operations $\odot$ and $\Rightarrow$ defined by
\begin{align*}
      (x,y)\odot(z,v) & :=(x\cdot z,v\rightarrow y), \\
(x,y)\Rightarrow(z,v) & :=(x\rightarrow z,v\cdot y)
\end{align*}
for all $(x,y),(z,v)\in Q^2$ and the constant $(1,1)$. Then $(Q,\leq,\cdot,\rightarrow,1)$ is a left-residuated groupoid if and only if $(Q^2,\leq,\odot,\Rightarrow,$ $(1,1))$ has this property.
\end{corollary}

\begin{proof}
This is the special case of Theorem~\ref{th2} where $a=b=1$, $f$ is the first and $g$ the second projection.
\end{proof}

\begin{corollary}
Let $(Q,\leq,\cdot,\rightarrow,1)$ be a poset with   binary operations $\cdot$ and $\rightarrow$ and a constant $1$ and consider the full twist product $(Q^2,\leq,\odot,\Rightarrow,(1,1))$ of $(Q,\leq)$ with binary operations $\odot$ and $\Rightarrow$ defined by
\begin{align*}
      (x,y)\odot(z,v) & :=(x\cdot v,z\rightarrow y), \\
(x,y)\Rightarrow(z,v) & :=(y\rightarrow z,v\cdot x)
\end{align*}
for all $(x,y),(z,v)\in Q^2$ and the constant $(1,1)$. Then $(Q,\leq,\cdot,\rightarrow,1)$ is a left-residuated groupoid if and only if $(Q^2,\leq,\odot,\Rightarrow,$ $(1,1))$ has this property.
\end{corollary}

\begin{proof}
This is the special case of Theorem~\ref{th2} where $a=b=1$, $f$ is the second and $g$ the first projection.
\end{proof}

\section{Operator residuated posets}

One can easily see that the left-residuated groupoid $(Q^2,\leq,\odot,\Rightarrow,(a,b))$ from Theorem~\ref{th2} need neither be commutative nor associative even if $(Q,\leq,\cdot,\rightarrow,1)$ has this property. Hence, we define the next concept.

\begin{definition}\label{def1}
An {\em operator residuated poset} is an ordered six-tuple $(Q,\leq,\odot,\Rightarrow,0,1)$ such that
\begin{enumerate}[{\rm(i)}]
\item $(Q,\leq,0,1)$ is a bounded poset,
\item $\odot$ and $\Rightarrow$ are mappings from $Q^2$ to $2^Q$ {\rm(}so-called {\em operators}{\rm)},
\item $x\odot y\approx y\odot x$,
\item $\bigcup\limits_{u\in x\odot y}(u\odot z)=\bigcup\limits_{u\in y\odot z}(x\odot u)$ {\rm(}{\em operator associativity}{\rm)},
\item $x\odot y\leq z$ if and only if $x\leq y\Rightarrow z$
\end{enumerate}
for all $x,y,z\in Q$.
\end{definition}

The following result shows that when using an operator residuated structure on the full twist product, commutativity and associativity of the original bounded left-residuated groupoid are preserved.

\begin{theorem}\label{th1}
Let $(Q,\leq,\cdot,\rightarrow,0,1)$ be a bounded commutative resid\-u\-at\-ed monoid and $a_0\in Q$. Then $(Q^2,\leq,\odot,\Rightarrow,(0,1),(1,0))$ where $(Q^2,\leq)$ is the full twist product of $(Q,\leq)$ and the operators $\odot$ and $\Rightarrow$ on $Q^2$ are defined by
\begin{align*}
      (x,y)\odot(z,v) & :=\{(x\cdot z,x\rightarrow v),(x\cdot z,z\rightarrow y)\}, \\
(x,y)\Rightarrow(z,v) & :=\{(x\rightarrow z,x\cdot v),(v\rightarrow y,x\cdot v)\}
\end{align*}
for all $(x,y),(z,v)\in Q^2$ is an operator residuated poset and the mapping $x\mapsto(x,a_0)$ an embedding of $(Q,\leq)$ into $(Q^2,\leq)$.
\end{theorem}

\begin{proof}
Let $a,b,c,d,e,f\in Q$. According to Lemmas~\ref{lem2}, \ref{lem3} and \ref{lem1}, $(Q,\leq,\cdot,\rightarrow,0,1)$ satisfies (1) -- (10).
\begin{enumerate}[(i)]
\item It is evident that $(Q^2,\leq,(0,1),(1,0))$ is a bounded poset.
\item $\odot$ and $\Rightarrow$ are mappings from $(Q^2)^2$ to $2^{(Q^2)}$.
\end{enumerate}
We must prove (iii) -- (v) of Definition~\ref{def1}.
\begin{enumerate}
\item[(iii)] We have
\begin{align*}
(a,b)\odot(c,d) & =\{(a\cdot c,a\rightarrow d),(a\cdot c,c\rightarrow b)\}=\{(c\cdot a,c\rightarrow b),(c\cdot a,a\rightarrow d)\}= \\
                & =(c,d)\odot(a,b).
\end{align*}
\item[(iv)] We have
\begin{align*}
\bigcup_{(x,y)\in(a,b)\odot(c,d)}((x,y)\odot(e,f)) & =\bigcup_{(x,y)\in\{(a\cdot c,a\rightarrow d),(a\cdot c,c\rightarrow b)\}}((x,y)\odot(e,f))= \\
                                                   & =((a\cdot c,a\rightarrow d)\odot(e,f))\cup((a\cdot c,c\rightarrow b)\odot(e,f))= \\
                                                   & =\{((a\cdot c)\cdot e,(a\cdot c)\rightarrow f),((a\cdot c)\cdot e,e\rightarrow(a\rightarrow d)), \\
                                                   & \hspace*{6mm}((a\cdot c)\cdot e,(a\cdot c)\rightarrow f),((a\cdot c)\cdot e,e\rightarrow(c\rightarrow b))\}= \\
                                                   & =\{(a\cdot(c\cdot e),a\rightarrow(c\rightarrow f)),(a\cdot(c\cdot e),(c\cdot e)\rightarrow b), \\
                                                   & \hspace*{6mm}(a\cdot(c\cdot e),a\rightarrow(e\rightarrow d)),(a\cdot(c\cdot e),(c\cdot e)\rightarrow b)\}= \\
                                                   & =((a,b)\odot(c\cdot e,c\rightarrow f))\cup((a,b)\odot(c\cdot e,e\rightarrow d))= \\
                                                   & =\bigcup_{(x,y)\in\{(c\cdot e,c\rightarrow f),(c\cdot e,e\rightarrow d)\}}((a,b)\odot(x,y))= \\
                                                   & =\bigcup_{(x,y)\in(c,d)\odot(e,f)}((a,b)\odot(x,y)).
\end{align*}
\item[(v)] The following are equivalent: 
\begin{align*}
& (a,b)\odot(c,d)\leq(e,f), \\
& \{(a\cdot c,a\rightarrow d),(a\cdot c,c\rightarrow b)\}\leq(e,f), \\
& a\cdot c\leq e,f\leq a\rightarrow d\text{ and }f\leq c\rightarrow b, \\
& a\leq c\rightarrow e,a\leq f\rightarrow d\text{ and }c\cdot f\leq b, \\
& (a,b)\leq\{(c\rightarrow e,c\cdot f),(f\rightarrow d,c\cdot f)\}, \\
& (a,b)\leq(c,d)\Rightarrow(e,f).
\end{align*}
\end{enumerate}
Finally, $(a,a_0)\leq(b,a_0)$ is equivalent to $a\leq b$.
\end{proof}

\begin{example}
If $(Q,\leq,\cdot,\rightarrow,0,1):=(\{0,1\},\leq,\cdot,(x,y)\mapsto1-x+xy,0,1)$ {\rm(}where $+$, $-$ and $\cdot$ denote addition, subtraction and multiplication of the reals, respectively{\rm)} then the tables for $\odot$ and $\Rightarrow$ look as follows:
\[
\begin{array}{c|c|c|c|c}
\odot &      (0,0)      &   (0,1)   &      (1,0)      &      (1,1) \\
\hline
(0,0) &    \{(0,1)\}    & \{(0,1)\} & \{(0,0),(0,1)\} & \{(0,0),(0,1)\} \\
(0,1) &    \{(0,1)\}    & \{(0,1)\} &    \{(0,1)\}    &    \{(0,1)\} \\
(1,0) & \{(0,0),(0,1)\} & \{(0,1)\} &    \{(1,0)\}    & \{(1,0),(1,1)\} \\
(1,1) & \{(0,0),(0,1)\} & \{(0,1)\} & \{(1,0),(1,1)\} &    \{(1,1)\}
\end{array}
\]
\[
\begin{array}{c|c|c|c|c}
\Rightarrow &      (0,0)      &      (0,1)      &   (1,0)   &      (1,1) \\
\hline
   (0,0)    &    \{(1,0)\}    & \{(0,0),(1,0)\} & \{(1,0)\} & \{(0,0),(1,0)\} \\
   (0,1)    &    \{(1,0)\}    &    \{(1,0)\}    & \{(1,0)\} &    \{(1,0)\} \\
   (1,0)    & \{(0,0),(1,0)\} &    \{(0,1)\}    & \{(1,0)\} & \{(0,1),(1,1)\} \\
   (1,1)    & \{(0,0),(1,0)\} & \{(0,1),(1,1)\} & \{(1,0)\} &    \{(1,1)\}
\end{array}
\]
\end{example}

\section{Pseudo-Kleene posets}

It was shown by the authors in \cite{CL} that every poset $\mathbf Q=(Q,\leq)$ can be embedded into a pseudo-Kleene one. For this we use a certain modification of the full twist product construction as follows.

Let $a\in Q$ and define
\begin{align*}
P_a(\mathbf Q) & :=\{(x,y)\in Q^2\mid L(x,y)\leq a\leq U(x,y)\}, \\
(x,y)\leq(z,v) & :\Leftrightarrow(x\leq z\text{ and }v\leq y), \\
        (x,y)' & :=(y,x)
\end{align*}
for all $(x,y),(z,v)\in Q^2$. The following was proved in \cite{CL}:
\begin{itemize}
\item $(P_a(\mathbf Q),\leq,{}')$ is a pseudo-Kleene poset,
\item the mapping $x\mapsto(x,a)$ is an embedding of $\mathbf Q$ into $(P_a(\mathbf Q),\leq)$,
\item $(P_a(\mathbf Q),\leq,{}')$ is a Kleene poset if and only if $\mathbf Q$ is distributive.
\end{itemize}

Since $P_a(\mathbf Q)$ is a subset of the full twist product of $\mathbf Q$, it is a question if residuation from $(Q,\leq,\cdot,\rightarrow,1)$ can be transferred to $P_a(\mathbf Q)$ as shown in Theorem~\ref{th2}. Unfortunately, this is not possible in general since $P_a(\mathbf Q)$ need not be closed under the operators $\odot$ and $\Rightarrow$ defined in Theorem~\ref{th2}. However, we can get necessary and sufficient conditions under which $P_a(\mathbf Q)$ is closed under these operators and hence becomes a pseudo-Kleene operator residuated poset.

If $\mathbf Q=(Q,\leq)$ is a poset, $a,b\in Q$ and every element of $P_a(\mathbf Q)$ is comparable with $(a,a)$ then $(a,b)\in P_a(\mathbf Q)$ and hence $(a,a)\leq(a,b)$ or $(a,b)\leq(a,a)$ whence $b\leq a$ or $a\leq b$, i.e.\ $b$ is comparable with $a$.

\begin{theorem}\label{th3}
Let $(Q,\leq,\cdot,\rightarrow,0,1)$ be a bounded commutative resid\-u\-at\-ed monoid and $a\in Q$ with $a\cdot a=a$, put $\mathbf Q:=(Q,\leq)$ and assume that all elements of $P_a(\mathbf Q)$ are comparable with $(a,a)$. Then
\begin{itemize}
\item $(P_a(\mathbf Q),\leq,\odot,\Rightarrow,(0,1),(1,0))$ where the operators $\odot$ and $\Rightarrow$ are defined by
\begin{align*}
      (x,y)\odot(z,v) & :=\{(x\cdot z,x\rightarrow v),(x\cdot z,z\rightarrow y)\}, \\
(x,y)\Rightarrow(z,v) & :=\{(x\rightarrow z,x\cdot v),(v\rightarrow y,x\cdot v)\}
\end{align*}
for all $(x,y),(z,v)\in P_a(\mathbf Q)$ is an operator residuated poset if and only if the following two conditions hold:
\begin{enumerate}
\item[{\rm(11)}] $a\cdot x<a$ implies $a\cdot x=0$,
\item[{\rm(12)}] $a<x$ implies $x\rightarrow a=a$.
\end{enumerate}
\item $(P_a(\mathbf Q),\leq,{}')$ where $(x,y)':=(y,x)$ for all $(x,y)\in P_a(\mathbf Q)$ is a pseudo-Kleene poset.
\item The mapping $x\mapsto(x,a)$ is an embedding of $\mathbf Q$ into $(P_a(\mathbf Q),\leq)$.
\item $(x,y)'\in(x,y)\Rightarrow(0,1)$ for all $(x,y)\in P_a(\mathbf Q)$.
\end{itemize}
\end{theorem}

\begin{proof}
Let $b,c,d,e\in Q$. According to Lemmas~\ref{lem2}, \ref{lem3} and \ref{lem1}, $(Q,\leq,\cdot,\rightarrow,0,1)$ satisfies (1) -- (10). If $a\leq x,y$ then $a=a\cdot a\leq a\cdot y\leq x\cdot y$ according to (1) and (2), i.e.\ we have
\begin{enumerate}
\item[(13)] $a\leq x,y$ implies $a\leq x\cdot y$.
\end{enumerate}
\begin{itemize}
\item Assume $(b,c),(d,e)\leq(a,a)$. Then the following hold: \\
$(b\cdot d,b\rightarrow e)\leq(a,a)$ because of (7) and (10). \\
$(b\cdot d,d\rightarrow c)\leq(a,a)$ because of (7) and (10). \\
Since $b\leq a$ we have $b\cdot e\leq a$ according to (7). \\
If $b\cdot e=a$ then $(b\rightarrow d,b\cdot e)$ is comparable with $(a,a)$. \\
If $b\cdot e<a$ then $(b\rightarrow d,b\cdot e)$ is comparable with $(a,a)$ if and only if $a\leq b\rightarrow d$. \\
$(e\rightarrow c,b\cdot e)\geq(a,a)$ because of (7) and (10).
\item Assume $(b,c)\leq(a,a)\leq(d,e)$. Then the following hold: \\
Since $b\leq a$ we have $b\cdot d\leq a$ according to (7). \\
If $b\cdot d=a$ then $(b\cdot d,b\rightarrow e)$ is comparable with $(a,a)$. \\
If $b\cdot d<a$ then $(b\cdot d,b\rightarrow e)$ is comparable with $(a,a)$ if and only if $a\leq b\rightarrow e$. \\
$(b\cdot d,d\rightarrow c)\leq(a,a)$ because of (7) and (10). \\
$(b\rightarrow d,b\cdot e)\geq(a,a)$ because of (7) and (10). \\
$(e\rightarrow c,b\cdot e)\geq(a,a)$ because of (7) and (10).
\item Assume $(d,e)\leq(a,a)\leq(b,c)$. Then the following hold: \\
$(b\cdot d,b\rightarrow e)\leq(a,a)$ because of (7) and (10). \\
Since $d\leq a$ we have $b\cdot d\leq a$ according to (7). \\
If $b\cdot d=a$ then $(b\cdot d,d\rightarrow c)$ is comparable with $(a,a)$. \\
If $b\cdot d<a$ then $(b\cdot d,d\rightarrow c)$ is comparable with $(a,a)$ if and only if $a\leq d\rightarrow c$. \\
Since $a\leq b,e$ we have $a\leq b\cdot e$ according to (13). \\
If $b\cdot e=a$ then $(b\rightarrow d,b\cdot e)$ is comparable with $(a,a)$. \\
If $a<b\cdot e$ then $(b\rightarrow d,b\cdot e)$ is comparable with $(a,a)$ if and only if $b\rightarrow d\leq a$. \\
If $b\cdot e=a$ then $(e\rightarrow c,b\cdot e)$ is comparable with $(a,a)$. \\
If $a<b\cdot e$ then $(e\rightarrow c,b\cdot e)$ is comparable with $(a,a)$ if and only if $e\rightarrow c\leq a$.
\item Assume $(a,a)\leq(b,c),(d,e)$. Then the following hold: \\
Since $a\leq b,d$ we have $a\leq b\cdot d$ according to (13). \\
If $b\cdot d=a$ then $(b\cdot d,b\rightarrow e)$ is comparable with $(a,a)$. \\
If $a<b\cdot d$ then $(b\cdot d,b\rightarrow e)$ is comparable with $(a,a)$ if and only if $b\rightarrow e\leq a$. \\
If $b\cdot d=a$ then $(b\cdot d,d\rightarrow c)$ is comparable with $(a,a)$. \\
If $a<b\cdot d$ then $(b\cdot d,d\rightarrow c)$ is comparable with $(a,a)$ if and only if $d\rightarrow c\leq a$. \\
$(b\rightarrow d,b\cdot e)\geq(a,a)$ because of (7) and (10), \\
Since $e\leq a$ we have $b\cdot e\leq a$ according to (7). \\
If $b\cdot e=a$ then $(e\rightarrow c,b\cdot e)$ is comparable with $(a,a)$. \\
If $b\cdot e<a$ then $(e\rightarrow c,b\cdot e)$ is comparable with $(a,a)$ if and only if $a\leq e\rightarrow c$.
\end{itemize}
Hence $(x,y)\odot(z,v)\subseteq P_a(\mathbf Q)$ and $(x,y)\Rightarrow(z,v)\subseteq P_a(\mathbf Q)$ for all $(x,y),(z,v)\in P_a(\mathbf Q)$ if and only if the following statements hold:
\begin{enumerate}[(a)]
\item $b,d\leq a\leq c,e$ and $b\cdot e<a$ imply $a\leq b\rightarrow d$.
\item $b,e\leq a\leq c,d$ and $b\cdot d<a$ imply $a\leq b\rightarrow e$.
\item $c,d\leq a\leq b,e$ and $b\cdot d<a$ imply $a\leq d\rightarrow c$.
\item $c,d\leq a\leq b,e$ and $a<b\cdot e$ imply $b\rightarrow d\leq a$.
\item $c,d\leq a\leq b,e$ and $a<b\cdot e$ imply $e\rightarrow c\leq a$.
\item $c,e\leq a\leq b,d$ and $a<b\cdot d$ imply $b\rightarrow e\leq a$.
\item $c,e\leq a\leq b,d$ and $a<b\cdot d$ imply $d\rightarrow c\leq a$.
\item $c,e\leq a\leq b,d$ and $b\cdot e<a$ imply $a\leq e\rightarrow c$.
\end{enumerate}
Now (a) is equivalent to the following statements:
\begin{align*}
& b\cdot a<a\text{ implies }a\leq b\rightarrow0, \\
& a\cdot b<a\text{ implies }a\cdot b\leq0, \\
& a\cdot b<a\text{ implies }a\cdot b=0, \\
& (11).
\end{align*}
In the same way one can see that (b), (c) and (h) are equivalent to (11). Moreover, (d) is equivalent to the following statements:
\begin{align*}
& a\leq b,e\text{ and }a<b\cdot e\text{ imply }b\rightarrow a\leq a, \\
& a<b\text{ implies }b\rightarrow a\leq a, \\
& a<b\text{ implies }b\rightarrow a=a, \\
& (12).
\end{align*}
In the same way one can see that (e), (f) and (g) are equivalent to (12). Moreover, we have
\[
(b,c)'=(c,b)\in\{(b\rightarrow0,b),(c,b)\}=\{(b\rightarrow0,b\cdot1),(1\rightarrow c,b\cdot1)\}=(b,c)\Rightarrow(0,1).
\]
The rest follows from Theorem~\ref{th1}.
\end{proof}

\begin{example}
Consider the bounded commutative residuated semigroup $(Q,\leq,\cdot,\rightarrow,0,$ $1)$ with $Q=\{0,a,1\}$, $0<a<1$ and
\[
\begin{array}{c|ccc}
\cdot & 0 & a & 1 \\
\hline
  0   & 0 & 0 & 0 \\
  a   & 0 & a & a \\
  1   & 0 & a & 1
\end{array}
\quad\quad\quad
\begin{array}{c|ccc}
\rightarrow & 0 & a & 1 \\
\hline
     0      & 1 & 1 & 1 \\
     a      & 0 & 1 & 1 \\
     1      & 0 & a & 1
\end{array}
\]
and put $\mathbf Q:=(Q,\leq)$. It is easy to check that $\mathbf Q$ is a distributive lattice and $(Q,\leq,\cdot,\rightarrow,0,1)$ is a bounded commutative resid\-u\-at\-ed monoid satisfying all the assumptions of Theorem~\ref{th3}. The poset $(P_a(\mathbf Q),\leq)$ is depicted in Figure~2:

\vspace*{-2mm}

\begin{center}
\setlength{\unitlength}{7mm}
\begin{picture}(4,10)
\put(2,1){\circle*{.3}}
\put(1,3){\circle*{.3}}
\put(3,3){\circle*{.3}}
\put(2,5){\circle*{.3}}
\put(1,7){\circle*{.3}}
\put(3,7){\circle*{.3}}
\put(2,9){\circle*{.3}}
\put(2,1){\line(-1,2)1}
\put(2,1){\line(1,2)1}
\put(1,3){\line(1,2)2}
\put(3,3){\line(-1,2)2}
\put(2,9){\line(-1,-2)1}
\put(2,9){\line(1,-2)1}
\put(1.35,.3){$(0,1)$}
\put(-.6,2.85){$(0,a)$}
\put(3.3,2.85){$(a,1)$}
\put(2.3,4.85){$(a,a)$}
\put(-.6,6.85){$(1,a)$}
\put(3.3,6.85){$(a,0)$}
\put(1.35,9.35){$(1,0)$}
\put(1.2,-.75){{\rm Fig.\ 2}}
\end{picture}
\end{center}

\vspace*{4mm}

Then the operators $\odot$ and $\Rightarrow$ have the following tables:
\[
\begin{array}{c|c|c|c|c|c|c|c}
\odot & 0a & 01 & a0 & aa & a1 & 10 & 1a \\
\hline
  0a  &  01   & 01 &  01   &  01   &  01   & 0a,01 & 0a,01 \\
  01  &  01   & 01 &  01   &  01   &  01   &  01   &  01 \\
  a0  &  01   & 01 &  a0   & a0,a1 & a0,a1 &  a0   & a0,a1 \\
  aa  &  01   & 01 & a0,a1 &  a1   &  a1   & a0,aa & aa,a1 \\
  a1  &  01   & 01 & a0,a1 &  a1   &  a1   & a0,a1 &  a1 \\
  10  & 0a,01 & 01 &  a0   & a0,aa & a0,a1 &  10   & 10,1a \\
  1a  & 0a,01 & 01 & a0,a1 & aa,a1 &  a1   & 10,1a &  1a
\end{array}
\]
\[
\begin{array}{c|c|c|c|c|c|c|c}
\Rightarrow &  0a   &  01   &  a0   &  aa   &  a1   & 10 &  1a \\
\hline
     0a     &  10   & a0,10 &  10   &  10   & a0,10 & 10 &  10 \\
     01     &  10   &  10   &  10   &  10   &  10   & 10 &  10 \\
     a0     &  0a   &  0a   &  10   & 0a,1a & 0a,1a & 10 & 0a,1a \\
     aa     & 0a,1a & 0a,aa &  10   &  1a   & aa,1a & 10 &  1a \\
     a1     & 0a,1a & 0a,1a &  10   &  1a   &  1a   & 10 &  1a \\
     10     &  0a   &  01   & a0,10 & 0a,aa & 01,a1 & 10 & 0a,1a \\
     1a     & 0a,1a & 01,a1 & a0,10 & aa,1a &  a1   & 10 &  1a
\end{array}
\]
Hence $(P_a(\mathbf Q),\leq,\odot,\Rightarrow,(0,1),(1,0))$ is an operator residuated poset $(P_a(\mathbf Q),\leq,{}')$ is a Kleene lattice.
\end{example}

\begin{example}
On the other hand, the bounded residuated monoid $(Q,\leq,\cdot,\rightarrow,1)$ from Example~\ref{ex1} has only two idempotents, namely $0$ and $1$. Every element of $P_0(\mathbf Q)=(\{0\}\times Q)\cup(Q\times\{0\})$ is comparable with $(0,0)$. But if $x\neq1$ then $0<x$, but $x\rightarrow0\neq0$ contradicting {\rm(12)}. Similarly, every element of $P_1(\mathbf Q)=(\{1\}\times Q)\cup(Q\times\{1\})$ is comparable with $(1,1)$. But if $x\neq0,1$ then $0<1\cdot x<1$ contradicting (11).
\end{example}

Authors' addresses:

Ivan Chajda \\
Palack\'y University Olomouc \\
Faculty of Science \\
Department of Algebra and Geometry \\
17.\ listopadu 12 \\
771 46 Olomouc \\
Czech Republic \\
ivan.chajda@upol.cz

Helmut L\"anger \\
TU Wien \\
Faculty of Mathematics and Geoinformation \\
Institute of Discrete Mathematics and Geometry \\
Wiedner Hauptstra\ss e 8-10 \\
1040 Vienna \\
Austria, and \\
Palack\'y University Olomouc \\
Faculty of Science \\
Department of Algebra and Geometry \\
17.\ listopadu 12 \\
771 46 Olomouc \\
Czech Republic \\
helmut.laenger@tuwien.ac.at

\end{document}